\documentclass[letterpaper, 10 pt, conference]{ieeeconf}  % Comment this line out if you need a4paper

\IEEEoverridecommandlockouts                              % This command is only needed if 
\usepackage{hyperref}
\usepackage{graphicx,color}
\graphicspath{{./IMG/}{images/}}
\usepackage{amsmath}
\usepackage{amssymb}
\usepackage{mathrsfs}
\usepackage{subfigure}
\usepackage{url}
\usepackage{booktabs}
\usepackage{array}
\usepackage[table]{xcolor}
\usepackage{mathtools} 		% needed for cases environment
\usepackage{epstopdf}
\usepackage{cite}
\usepackage[vlined,ruled]{algorithm2e}
\usepackage{upgreek}
\usepackage{dsfont}
\usepackage{comment}

\newtheorem{theorem}{Theorem}[section]
\newtheorem{lemma}[theorem]{Lemma}

\newtheorem{proposition}[theorem]{Proposition}

\newtheorem{remark}{Remark}
\newtheorem{assumption}{Assumption}

\newcommand{\mc}{\mathcal}

\newcommand{\real}{\mathbb{R}}

\newcommand{\integnneg}{\mathbb{Z}_{\geq 0}}
\newcommand{\complex}{\mathbb{C}}
\newcommand{\tsp}{\mathsf{T}} 
 
\newcommand{\inv}{{\negat 1}} 
\newcommand{\negat}{\scalebox{0.75}[.9]{\( - \)}}

\newcommand*{\QEDB}{\hfill\ensuremath{\square}}%  empty square
\newcommand*{\QEDBL}{\hfill\ensuremath{\blacksquare}}% Black square
\newcommand\oprocendsymbol{\hbox{$\square$}}
\newcommand\oprocend{\relax\ifmmode\else\unskip\hfill%
\fi\oprocendsymbol}

\newcommand{\setdef}[2]{\{#1 \; : \; #2\}}

\newcommand{\sps}[2]{{#1}^{\textup{#2}}}

\newcommand{\until}[1]{\{1,\dots,#1\}}

\newcommand{\norm}[1]{\Vert #1 \Vert}

\title{\LARGE \bf
Optimization of Linear Multi-Agent Dynamical Systems \\ 
via Feedback Distributed Gradient Descent  Methods}

\author{Amir Mehrnoosh and Gianluca Bianchin\thanks{
The authors are with ICTEAM Institute and the Department of Mathematical Engineering at UCLouvain, Belgium. A. Mehrnoosh is supported by F.R.S.-FNRS.
\{\href{mailto:amir.mehrnoosh@uclouvain.be}{\texttt{amir.mehrnoosh}},
    \href{mailto:gianluca.bianchin@uclouvain.be}{\texttt{gianluca.bianchin\}@uclouvain.be}.}
    This work was supported in part by the ARC project `SIDDARTA'.
}}

\begin{document}

\maketitle
\thispagestyle{empty}
\pagestyle{empty}

%%%%%%%%%%%%%%%%%%%%%%%%%%%%%%%%%%%%%%%%%%%%%%%%%%%%%%%%%%%%%%%%%%%%%%%%%%%%%%%%
\begin{abstract}

Feedback optimization is an increasingly popular control paradigm to optimize dynamical systems, accounting for control objectives that concern the system operation at steady-state. Existing feedback optimization techniques heavily rely on centralized systems and controller architectures, and thus suffer from scalability and privacy issues when systems become large-scale. In this paper, we propose a distributed architecture for feedback optimization inspired by distributed gradient descent, whereby each agent updates its local control variable by combining the average of its neighbors with a local negative gradient step. Under convexity and smoothness assumptions for the cost, we establish convergence of the control method to a critical optimization point. By reinforcing the assumptions to restricted strong convexity, we show that our algorithm converges linearly to a neighborhood of the optimal point, where the size of the neighborhood depends on the choice of the stepsize. Simulations corroborate the theoretical results.

\end{abstract}

%%%%%%%%%%%%%%%%%%%%%%%%%%%%%%%%%%%%%%%%%%%%%%%%%%%%%%%%%%%%%%%%%%%%%%%%%%%%%%%%
\section{Introduction}

Optimal steady-state regulation is concerned with the problem of controlling a 
dynamical systems to an optimal steady-state point, as characterized by a 
mathematical optimization problem~\cite{AJ-ML-PB:09}. The classical approach 
to tackle this goal relies on the principle of separation between planning and 
control, whereby the optimization problem is solved beforehand (offline) to 
determine optimal system states, which are then inputed as references to 
controllers responsible for regulating the system to these states. 
Remarkably, a key assumption in this approach is that disturbances are known 
beforehand and fed to the optimization solver; this allows the optimization 
to be solved with high precision to generate the required reference states. 
Unfortunately, in most control applications, disturbances are unknown. Often, the main objective of a control system is to ensure optimality in 
the face of unmeasurable disturbances or imprecise system knowledge. 
Notably, classical 
batch optimization algorithms fail~\cite{ED-AS-SB-LM:20} when are 
approximately known or vary after the optimization has been solved because disturbances may perturb optimal steady states.

Recently, several authors have studied optimal steady-state regulation 
problems in a centralized setting. A list of representative works on this 
topic includes~\cite{MC-ED-AB:20,AH-SB-GH-FD:20,SM-AH-SB-GH-FD:18,LL-JS-EM:18,GB-JC-JP-ED:21-tcns,GB-JIP-ED:20-automatica,FB-HD-CE:12}. See also the recent developments using zeroth order 
algorithms and data-driven approaches~\cite{ZH-SB-JH-FD-XG:23,XC-JP-NL:22,GB-MV-JC-ED:24-tac}.
Feedback optimization controllers have gained popularity thanks to their 
capability to regulate physical systems to optimal steady-state points while 
rejecting constant~\cite{AH-SB-GH-FD:20} or time-varying 
disturbances~\cite{GB-JC-JP-ED:21-tcns,GB-JIP-ED:20-automatica}.
The central idea consists of adapting numerical optimization algorithms to 
operate as feedback controllers. This is achieved by using an inexact gradient 
evaluated using real-time measurements to update control inputs without 
requiring the exact plant model and disturbances. This feature endows feedback optimization with the versatility to handle various scenarios. 
All of these methods are designed to be implemented in a centralized 
architecture, and thus suffer from scalability issues when systems become 
large-scale, as well as privacy concerns when cost functions or feedback 
signals need to be maintained private. This work departs from this existing 
literature by focusing on the problem of optimal steady-state regulation for 
systems with a distributed architecture. This connects our 
work with the body of literature on distributed optimization. Distributed gradient descent (DGD) was proposed in~\cite{AN-AO:07}, studied 
in~\cite{KY-QL-WY:16N}, a diminishing stepsize was used in~\cite{AN-AO-PAP:10} 
to ensure exact convergence;  
see also~\cite{JD-AA-MW:11,DJ-JX-JM:14,AN-AO:09,CX-VM-RX-EA-UK:18}. 
Other distributed optimization algorithms have been explored in recent years; we refer to~\cite{TY-XY-JW-etal:19} for a comprehensive discussion.
Particularly related to our problem are the works~\cite{TL-ZQ-YH-ZPJ:21,WW-ZH-GB-SB-FD:23,CC-MC-JC-ED:19,GC-GN:22}. Compare to~\cite{TL-ZQ-YH-ZPJ:21}, we do not require a two-layer control architecture and tracking controllers; in contrast to~\cite{WW-ZH-GB-SB-FD:23}, we do not approximate the system’s sensitivity matrix by its diagonal elements, ignoring the coupling between subsystems, which leads to a loss in accuracy; \cite{CC-MC-JC-ED:19} focuses on systems modeled as a static linear map, while in this work we account for dynamics; finally, with respect to~\cite{GC-GN:22}, we account for performance metrics that depend on a vector quantity as opposed to a scalar aggregate one. 

This work features three main contributions. First, we propose a distributed 
architecture for the optimal steady-state regulation problem, and a distributed control algorithm to address this problem. Our algorithm is inspired by distributed optimization approaches and combines a gradient-descent step with a consensus operation to simultaneously solve an optimization and seek an agreement between the agents. Second, we present proof of convergence to a fixed point for the controller-system state. Our technical arguments provide guidelines on how to choose the controller stepsize to guarantee convergence of the controlled system. 
Third, we provide an explicit bound for the control error. Precisely, we show that under restricted strong convexity assumptions, the controller state converges linearly to a neighborhood of the optimal point. In line with the existing literature~\cite{KY-QL-WY:16N}, the size of such a neighborhood depends on the choice of the controller stepsize. Intuitively, convergence cannot be exact since distributed controllers need to average between moving in a direction that decreases the gradient while maintaining an agreement with the other controllers. 

The rest of the article is organized as follows.
Section~\ref{sec:problem_formulation} formalizes the problem of distributed optimal steady-state regulation. Section~\ref{sec:control_design} illustrates the proposed controller. Section~\ref{sec:analysis} presents the two main contributions of this work. Section~\ref{sec:simulations} validates the findings through numerical simulations, and Section~\ref{sec:conclusions} concludes the paper.

\textbf{Notation.}
For a symmetric matrix $W$, we denote its eigenvalues by 
$\lambda_1(W) \geq \lambda_2(W) \geq \cdots  \geq \lambda_N(W).$
Given a symmetric and doubly stochastic matrix $W$, we let the 
eigenvalues be sorted in a nonincreasing order:
$1 = \lambda_1(W) \geq \lambda_2(W) \geq \cdots \geq \lambda_N(W) 
{>} -1.$ Moreover, we let
$\beta := \max \{ \vert \lambda_2(W) \vert, \vert \lambda_N(W)\vert \}.$
% \end{align}

\section{Problem Setting}

\label{sec:problem_formulation}

\begin{figure}[t]
\centering 
\includegraphics[width=0.9\columnwidth]{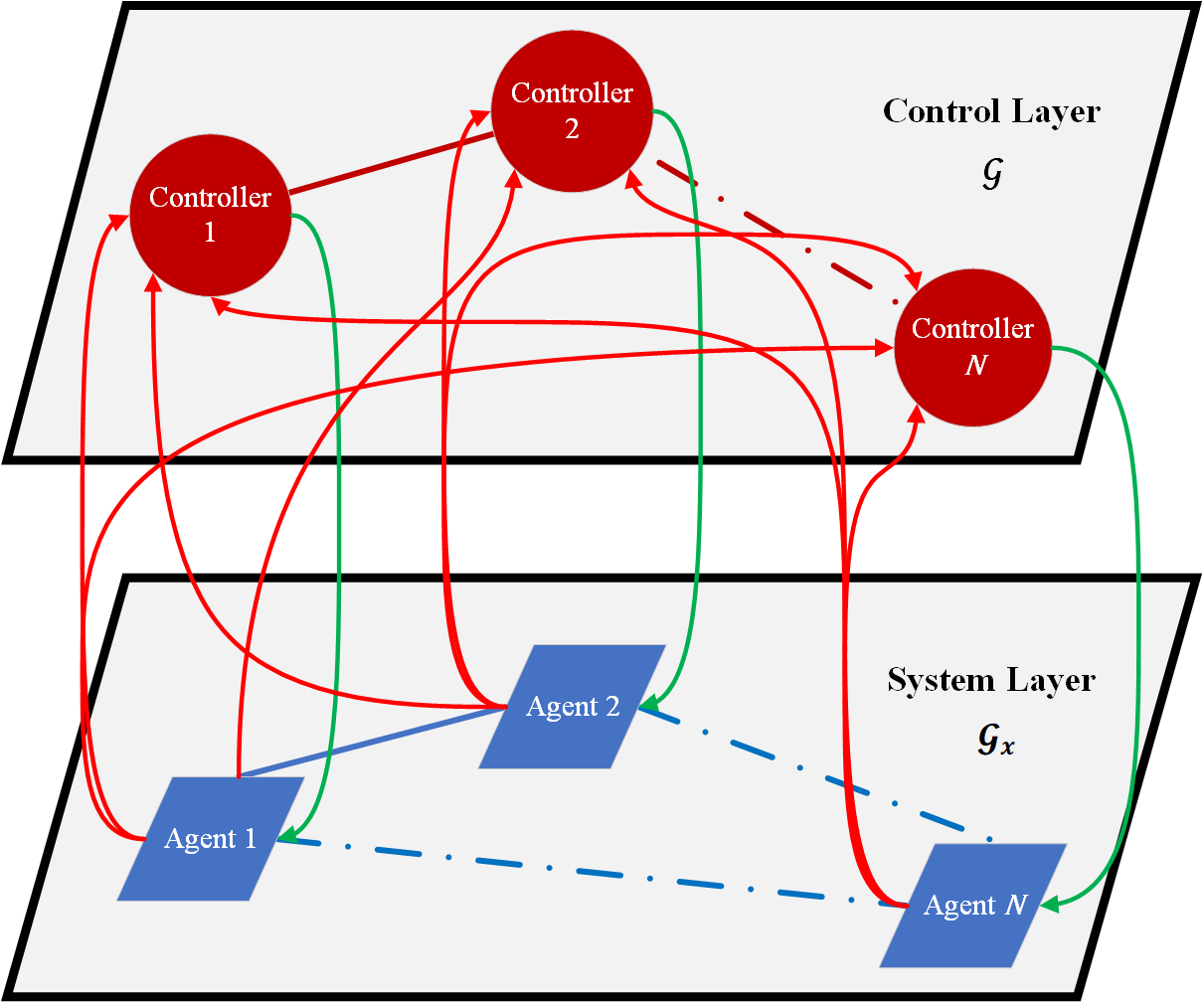}
\caption{Distributed system architecture considered in this work 
(cf.~\eqref{eq:plantModel_distributed}). Each local controller actuates 
the corresponding subsystem (green lines), by using global feedback 
information (red lines), see~\eqref{eq:distributed_control_algorithm}.}
\vspace{-.6cm}
\label{fig:architecture}
\end{figure}

\subsection{System to control}
Consider a dynamical system composed of $N$ subsystems, 
indexed by $\mc V_x = \until N;$ we describe the physical couplings 
between the subsystems using a directed graph (called \textit{system graph}) 
$\mc G_x =(\mc V_x, \mc E_x)$,  where 
$\mc E_x \subseteq \mc V_x \times \mc V_x.$
See Fig.~\ref{fig:architecture}-(System Layer).
Each subsystem $i \in \mc V$ is described by a local state 
$x_i^k \in \real^{n_i},$ updating as: 
\begin{align}
\label{eq:plantModel_distributed}
x_i^{k+1} &=  \sum_{j \in \mc N_i} A_{ij} x_j^{k} + B_i u_i^{k} + E_i q_i, && k \in \integnneg,
\end{align}
where $\mc N_i$ denotes the set of subsystems interacting with $i,$ $u_i^k \in \real^{m_i}$ is the local control decision, 
and $A_{ij} \in \real^{n_i \times n_j}, B\in \real^{n_i \times m_i}.$
In~\eqref{eq:plantModel_distributed}, $q_i \in \real^{r_i}$
with $E_i \in \real^{n_i \times r_i},$ models a constant unknown 
disturbance acting on subsystem $i$. 
We will denote by $n:=\sum_i n_i,$  $m:=\sum_i m_i,$ and $r:=\sum_i r_i.$
Further, we assume that the system's state is observed by means of 
a global output signal:
\begin{align}\label{eq:output_y}
y^k &= \sum_{i=1}^N C_i x_i^{k} + D_i u_i^k,
\end{align}
where $C_i \in \real^{p \times n_i}$ and 
$D_i \in \real^{p \times m_i}.$

\smallskip
\begin{remark}{\bf \textit{(Output model)}}
\label{rem:output_model}
Equation~\eqref{eq:output_y} assumes the availability of a 
common, or centralized, output signal. 
Although alternative output models could be considered (e.g., where 
each $i$ measures $y_i^k = C_ix_i^k+D_i u_i^k$), the 
model~\eqref{eq:output_y} is common to many practical applications. 
For example, consider a swarm of drones where each drone communicates wirelessly with a common base station.
% another example is that of industrial systems, where a Supervisory Control and Data Acquisition 
% (SCADA) system collects and processes all sensor measurements. 
    % \red{In many real-world multi-agent systems, the agents must coordinate to achieve a common objective while having limited or no access to one another's internal states or cost functions. However, achieving global coordination often requires a shared measurement that reflects the system’s overall performance. For example, in a swarm of drones collaboratively transporting a load, each drone is equipped with its own private sensing and control capabilities. Still, all drones rely on a shared, global measurement, such as the position of the load or the error relative to the target location. This global measurement is critical for allowing agents to coordinate without communicating their internal states or objectives.}
    \QEDB
\end{remark}

In vector form, \eqref{eq:plantModel_distributed} reads as:
\begin{align}
\label{eq:plantModel_global}
x^{k+1} &=  A x^{k} + B u^{k} + E q, \nonumber \\
y^k &= C x^{k} + D u^k,
\end{align}
where $x^k = (x_1^k, \dots, x_N^k) \in \real^n$ is the vector of states, 
$u^k = (u_1^k, \dots , u_N^k) \in \real^m$ the vector of inputs, 
% $y^k = (y_1^k, \dots , y_n^k)$ is the vector of outputs, 
$q = (q_1, \dots, q_N) \in \real^r$ the vector of disturbances,
$A=[A_{ij}],$ $B=[B_{ij}],$ $E=[E_{ij}],$ $C=[C_1, \dots, C_N],$ and $D=[D_1, \dots , D_N]$
are block matrices.
In what follows, we let $G(z) = C (zI - A)^\inv B + D $, $H(z) = C (zI - A)^\inv E$
% \begin{align}\label{eq:transfer_functions}
% G(z) &= C (zI - A)^\inv B + D, \nonumber \\
% H(z) &= C (zI - A)^\inv E,
% \end{align}
denote the transfer functions from $u$ to $y$ and from $q$ to $y$, 
respectively. This definition is intended for the 
values of $z \in \complex$ for which the inverse is defined.
We will adopt the compact notation $G:=G(1), H:=H(1)$ to express the 
steady-state response of the system to constant inputs; formally, when 
$u_k = \bar u, q_k = \bar q$ for all $k \in \integnneg$:
\begin{align*}
\lim_{k \to \infty} y_k = G\bar u + H \bar q.
\end{align*}
% \begin{align*}
% G:=G(1), && H:=H(1).
% \end{align*}

% We make the following assumption on~\eqref{eq:plantModel_distributed}.

\begin{assumption}{\bf\textit{(Stability and control properties of system)}}
\label{ass:stabilityPlant}
The model~\eqref{eq:plantModel_distributed} is asymptotically stable, 
controllable, and observable.
\QEDB
\end{assumption}

Thanks to Assumption~\ref{ass:stabilityPlant}, in what follows, we 
will fix a matrix $Q\succ 0$, and we let $P \succ 0$ be such that 
$A^\tsp P A - P = -Q$.
Controllability and observability are standard assumptions to guarantee 
that control problems are well-defined. Moreover, we assume that the system has been pre-stabilized as in 
Assumption~\ref{ass:stabilityPlant}. This can be achieved using 
well-established static state feedback techniques~\cite{HKK:96}. 

\subsection{Distributed structure of the controller}
\label{sec:distrib_control_structure}

We consider a system controlled by distributed controllers, each 
co-located with a local subsystem and actuating its corresponding control variable (see Fig.\ref{fig:architecture}). The combination of a subsystem and its controller is referred to as an agent. Controllers collaborate through an undirected \textit{control graph}
$\mc G=(\mc V_u, \mc E_u)$, where $\mc V_u=\mc V_x$ and $\mc E_u \subseteq \mc V_u \times \mc V_u$ (see Fig.\ref{fig:architecture}-Control Layer). 
A pair of controllers can collaboratively compute a control law only if they are connected by a link in $\mc E_u.$ Recall that a graph is connected if there exists a path between any two 
nodes.
\smallskip
\begin{assumption}{\bf\textit{(Connectivity of the control graph)}}
\label{ass:networkConnectivity}
The graph $\mc G$ is connected.
\QEDB\end{assumption}

Under Assumption~\ref{ass:networkConnectivity}, there exists a
symmetric and doubly  stochastic matrix $W =[w_{ij}] \in\real^{N \times N}$ (which will be called \textit{mixing matrix}) such that $(i,j)\not \in \mc E_u$ implies $w_{ji}=0$, and that satisfies  
$\beta <1$.

\subsection{Control objectives as an optimization problem}
We study a control problem where the ensemble of controllers seeks to 
collaboratively compute an input that solves
\begin{align}\label{eq:steady_state_optimization}
\underset{u \in \real^m}{\text{minimize}} 
~~\sum_{i=1}^N \Phi_i(u,Gu+Hq),
\end{align}
where $\Phi_i: \real^m \times \real^p \to \real.$ 
The optimization problem~\eqref{eq:steady_state_optimization} describes 
a setting where the group of controllers wants to determine a 
control input that optimizes (as measured by the cost 
$\sum_{i=1}^N \Phi_i(\cdot, \cdot)$) the system at steady-state 
(captured by the dependence of the cost on the steady-state output 
$Gu+Hq$).
Moreover, the cost in~\eqref{eq:steady_state_optimization} has a 
separable structure, allowing for cases where $\Phi_i(\cdot,\cdot)$ 
is known locally only by agent $i.$ We also remark that optimization 
problem~\eqref{eq:steady_state_optimization} is \textit{parametrized} 
by $q$; as such, its solutions cannot be computed using 
standard optimization solvers, since the disturbance $q$ is unknown 
and unmeasurable.

\smallskip
\begin{remark}{\bf \textit{(Practical relevance of~\eqref{eq:steady_state_optimization})}}
Allowing the local costs $\Phi_i(\cdot, \cdot)$ 
in~\eqref{eq:steady_state_optimization} to depend on the global system 
input $u$ and the global steady-state output $Gu+Hq$ enables our framework to model systems where agents have individual performance metrics (namely, $\Phi_i(\cdot, \cdot)$), but collectively optimize a shared objective, as in \eqref{eq:steady_state_optimization}. This applies to scenarios like energy systems, where each subsystem evaluates optimality differently. 
% (namely, $\Phi_i(\cdot, \cdot)$), but altogether the group seeks 
% to strike a balance between these heterogeneous objectives (by minimizing their sum, as in \eqref{eq:steady_state_optimization}). 
% Examples of this setting are energy systems, where each subsystem may 
% measures system optimality using a different performance measure. 

Moreover, notice that a special case 
of~\eqref{eq:steady_state_optimization} is:
\begin{align}\label{eq:steady_state_optimization_adjusted}
\underset{u \in \real^m}{\text{minimize}} 
~~\sum_{i=1}^N \tilde \Phi_i(u_i,Gu+Hq),
\end{align}
where $\tilde \Phi_i: \real^{m_i} \times \real^p \to \real$ now depends 
only on the local actuation variable $u_i$ (instead of the global one).
We stress that our framework is general enough to account for this 
setting as a special case. This formulation describes, for example, 
problems where the ensemble of agents would like to optimize the global 
system operation (as described by $Gu+Hq$), while minimizing the local 
control effort. 
Returning to the swarm of drones example (see 
Remark~\ref{rem:output_model}), each drone may seek to reduce its local 
power consumption, while ensuring that the entire swarm reaches a 
desired configuration, which is measured by the global $y$.
    \QEDB
\end{remark}

In the remainder, we will denote in compact form
$$\Phi(u,y)  :=  \sum_{i=1}^N \Phi_i(u,y).$$
\smallskip
\begin{assumption}{\bf\textit{(Lipschitz and convexity of the cost)}}
\label{ass:convexityLocalCost}
For all $i$,  $(u, y) \mapsto \Phi_i(u,y)$ is proper closed convex, 
lower bounded, and Lipschitz differentiable with constant $L_{\Phi_i}$. 
\QEDB\end{assumption}

Assumption~\ref{ass:convexityLocalCost} is standard in optimization. 
This assumption allows us to derive the following inequality\footnote{
The notation $\nabla \Phi(u,y)$ indicates $\nabla \Phi(u,y) = (\nabla_u \Phi(u,y), \nabla_y \Phi(u,y)) \in \real^{m + p}.$}:
\begin{align}\label{eq:lipschitz_gradient}
    \norm{\Pi^\tsp( \nabla \Phi(u,y) - \nabla \Phi(u',y'))} \leq L_\Phi \left \norm{\begin{bmatrix} u \\ y \end{bmatrix} - \begin{bmatrix} u' \\ y' \end{bmatrix} \right},
\end{align}
where $\Pi^\tsp := \begin{bmatrix} I_m &  G^\tsp \end{bmatrix},$
which holds for all $y, y' \in \real^n$, $u, u' \in \real^m$, and $L_\Phi := \norm{\Pi} \sum_i L_{\Phi_i}.$ In what follows, we denote the set of optimizers 
of~\eqref{eq:steady_state_optimization}~by
\begin{align*}
\mc A^* := \setdef{(u^*, x^*)}{(u^*, x^*) \text{ is a first-order optimizer of~\eqref{eq:steady_state_optimization}}},
\end{align*}
and we assume that this set is nonempty and closed.

\section{Design of the F-DGD Method}

\label{sec:control_design}
A centralized algorithm to solve the steady-state regulation 
problem~\eqref{eq:steady_state_optimization} has been studied 
in~\cite{GB-MV-JC-ED:24-tac} and continuous-time 
counterparts~\cite{AH-SB-GH-FD:20,GB-JIP-ED:20-automatica,GB-JC-JP-ED:21-tcns}. 
In~\cite{GB-MV-JC-ED:24-tac}, the authors propose a gradient-type controller
\begin{align}
\label{eq:ideal_controller}
u^{k+1}  = u^k -\eta \Pi^\tsp \nabla \Phi(u^k,y^k), 
\end{align}
where $\eta > 0$ denotes a scalar stepsize, being a design parameter.
The controller~\eqref{eq:ideal_controller} implements a gradient-type
iteration to solve the optimization~\eqref{eq:steady_state_optimization}, 
modified by replacing the true gradient $\Pi^\tsp \nabla \Phi (u^k,Gu^k+Hq)$ 
with a measurement-based version $\Pi^\tsp \nabla \Phi (u^k,y^k)$, which 
avoids the need to measure $q$. Unfortunately, \eqref{eq:ideal_controller} is inapplicable to our setting, since %(i) \eqref{eq:ideal_controller} does not respect the distributed nature of the controller considered here (cf. Section~\ref{sec:distrib_control_structure}), and (ii) \eqref{eq:ideal_controller} requires centralized knowledge of the gradients $\{\nabla \Phi_i \}_{i \in \mc V_u},$ which is impractical in our case since each $\Phi_i$ is known only locally.
\begin{enumerate}
    \item[(i)] \eqref{eq:ideal_controller} does not respect the distributed 
nature of the controller considered here (cf. Section~\ref{sec:distrib_control_structure}), and 
    \item[(ii)] \eqref{eq:ideal_controller} requires centralized
knowledge of the gradients $\{\nabla \Phi_i \}_{i \in \mc V_u},$ 
which is impractical in our case since each $\Phi_i$ is assumed to be known only locally. 
\end{enumerate}
With this motivation, we propose an algorithm where each agent 
$i \in \mc V_u$ holds a local copy 
$u^k_{(i)} \in \real^{m}$ of $u^k$, and updates it as:
\begin{align}\label{eq:distributed_control_algorithm}
    u_{(i)}^{k+1} = \sum_{j=1}^N w_{ij} u_{(j)}^{k} - \eta 
    \Pi^\tsp \nabla \Phi_i (u^k_{(i)},y^k),
\end{align}
where $w_{ij}$ are the entries of matrix $W.$
According to this control law, each agent $i$ updates its local state 
$u_{(i)}$ by performing two steps: it computes a weighted average of 
its neighbors' states $\sum_{j=1}^N w_{ij} u_{(j)}^{k}$ to seek a 
consensus between the agents, and it applies 
$-\Pi^\tsp \nabla \Phi_i (u^k_{(i)},y^k)$ to seek to decrease 
$\Phi_i (u^k_{(i)},y^k).$
We remark that this control law is distributed in the sense that each agent $i$ 
requires only knowledge of the local $\nabla \Phi_i$. Note that each agent needs to know the steady-state map (i.e., $G$) and measure the global output feedback 
signal $y^k$. We call~\eqref{eq:distributed_control_algorithm} \textit{Feedback Distributed Gradient Descent (F-DGD) algorithm.}

\begin{remark}{\bf \textit{(Relationship with distributed optimization algorithm)}} 
The algorithm~\eqref{eq:distributed_control_algorithm} is inspired by the Distributed Gradient Descent (DGD) method~\cite{KY-QL-WY:16N}. While other approaches, such as EXTRA and Gradient 
Tracking~\cite{TY-XY-JW-etal:19}, are valid alternative solutions, 
their investigation is left as the topic of future works.
\QEDB\end{remark}

\section{Convergence Analysis and Error Bounds}

\label{sec:analysis}
In this section, we study the convergence of~\eqref{eq:distributed_control_algorithm} 
when applied to control the system~\eqref{eq:plantModel_global}. 
In the remainder, we employ the following notations of stacked vectors:
$u^k_{(1:N)} := (u_{(1)}^k, u_{(2)}^k, \dots, u^k_{(N)}) \in \real^{mN}$ and 
\begin{align*}
\gamma(u^k_{(1:N)}, y^k) := 
\begin{bmatrix}
\Pi^\tsp \nabla \Phi_1 (u^k_{(1)},y^k)\\    
\vdots\\
\Pi^\tsp \nabla \Phi_N (u^k_{(N)},y^k)
\end{bmatrix}
\in \real^{mN}.
\end{align*}
In vector form, the system~\eqref{eq:plantModel_global} controlled 
by~\eqref{eq:distributed_control_algorithm} reads as
\begin{subequations}
\label{eq:interconnectedSystem}
\begin{align}
\label{eq:interconnectedSystem-a}
x^{k+1} &=  A x^{k} + BSu^{k}_{(1:N)} + E q,  \\
y^k &= C x^{k} + DS u^k_{(1:N)}, \nonumber\\
\label{eq:interconnectedSystem-b}
u_{(i)}^{k+1} & = \sum_{j \in \mc N_i} w_{ij} u_{(j)}^{k} - \eta 
    \Pi^\tsp \nabla \Phi_i (u^k_{(i)},y^k), i \in \mc V_u,
\end{align}
\end{subequations}
where $S \in \real^{m \times mN}$ is given by
$$S=\text{diag}([I_{m_1}, 0, \dots ], [0, I_{m_2}, 0, \dots ], \dots 
[0, \dots, 0, I_{m_N}]).$$

\subsection{Asymptotic convergence}
The following result shows that, under a suitable choice of the 
stepsize $\eta,$ the state of~\eqref{eq:interconnectedSystem} converges 
asymptotically.

\begin{theorem}{\bf \textit{(Convergence of the state sequences)}}
\label{thm:convergence}
Let Assumptions~\ref{ass:stabilityPlant}-\ref{ass:convexityLocalCost} 
hold, $W$ be such that $\beta<1,$ and $\eta \leq \bar \eta := \min \{ \eta_1, \eta_2, \eta_3\},$ with
\begin{align}\label{eq:def_etas}
  \eta_1 = \frac{1-2\mu + \lambda_N(W)}{L_\Phi}, &&  
  \eta_2 = \frac{\mu}{\lambda_1(P)L_h^2},
\end{align}
$$\eta_3 = \frac{\mu\lambda_n(Q)}{\frac{L_\Phi^2}{4} + L_h^2(\norm{A^\tsp P}^2 +\lambda_n(Q)\lambda_1(P)) + {L_h} L_\Phi\norm{A^\tsp P}},$$
with $\mu$ an arbitrary constant, $ 0< \mu \leq 1 - \frac{{(1- \lambda_N(W)) + \eta {L_{\Phi}}}}{2}$, and $L_h = \norm{(I-A)^\inv B S}$.
Then, the sequences $x^k,$ $u^{k}_{(i)}$ generated 
by~\eqref{eq:interconnectedSystem} converge. 
\QEDB\end{theorem}

\begin{proof}
% Please see~\cite{AM-GB:24} for the complete proof.
We will prove this claim by using La Salle's Invariance 
Principle~\cite[Thm 4.4]{HKK:96}. For clarity of presentation, the proof is organized 
into {four} main steps.

\textit{1) Change of variables and storage function.}
Let $h(u) = (I-A)^\inv B S u { + } (I-A)^\inv E q,$ and consider the 
new coordinate $\tilde x^k = x^k - h(u^k_{(1:N)}),$ which shifts the 
equilibrium point of~\eqref{eq:interconnectedSystem-a} to the origin. 
Inspired by singular-perturbation reasonings~\cite[Sec.~11]{HKK:96}, 
consider the storage function:
\begin{align}\label{eq:LyapunovFunctionDef}
U(u_{(1:N)}, \tilde x) := \frac{d}{\eta} V_u(u_{(1:N)}) + (1-d) V_x(\tilde x),
\end{align}
where $\tilde x \in \real^n, u_{(1:N)} \in \real^{mN},$ and 
$d \in (0,1)$. We let
\begin{align}\label{eq:Vu_Vx_definition}
V_u(u_{(1:N)}) &= -\frac{1}{2} \sum_{i,j=1}^N w_{ij} u_{(i)}^\tsp {u_{(j)}} \nonumber \\ 
&\quad + \sum_{i=1}^N \left( \frac{1}{2} \norm{u_{(i)}}^2 + \eta \Phi_i(u_{(i)},Gu_{(i)}+Hq) \right), \nonumber\\
V_x(\tilde x) &= \tilde x^\tsp P \tilde x. 
\end{align}
Notice that $V_u$ is Lipschitz differentiable with constant 
$L_{V_u} {\leq} (1- \lambda_N(W)) + \eta {L_{\Phi}}$ and 
it is convex (since all $\Phi_i$ are convex and 
$\sum_{i=1}^N  \norm{u_{(i)}}^2  - \sum_{i,j=1}^N w_{ij} u_{(i)}^\tsp u_{(j)}$ is also convex due to $\lambda_1(W)=1$).
Next, we introduce the quantity
\begin{align}\label{eq:Vu}
& \tilde V_u(u_{(1:N)},\tilde x) = -\frac{1}{2} \sum_{i,j=1}^N w_{ij} u_{(i)}^\tsp {u_{(j)}} \\ 
&\quad + \sum_{i=1}^N \left( \frac{1}{2} \norm{u_{(i)}}^2 + \eta \Phi_i(u_{(i)},C \tilde x+Gu_{(i)}+Hq) \right), \nonumber
\end{align}
and $F_c(u_{(1:N)},\tilde x) :=  
\begin{bmatrix}
        \nabla_{u_{(1)}} \tilde V_u(u_{(1:N)},\tilde x)\\  
        \vdots\\
        \nabla_{u_{(N)}} \tilde V_u(u_{(1:N)},\tilde x)
\end{bmatrix}.$
% \begin{align*}
% F_c(u_{(1:N)},\tilde x) :=  
% \begin{bmatrix}
%         \nabla_{u_{(1)}} \tilde V_u(u_{(1:N)},\tilde x)\\  
%         \vdots\\
%         \nabla_{u_{(N)}} \tilde V_u(u_{(1:N)},\tilde x)
% \end{bmatrix}.
% \end{align*}
With this notation, \eqref{eq:interconnectedSystem-b} and~\eqref{eq:Vu} 
can be re-expressed as:
\begin{align}\label{eq:controller_reformulation}
u^{k+1}_{(1:N)} &= u^{k}_{(1:N)} - F_c(u_{(1:N)}^k,\tilde x^k),\nonumber\\
V_u(u_{(1:N)}) &=\tilde  V_u(u_{(1:N)}, 0),
\end{align}
using $y^k = C(\tilde x^k+h(u_{(1:N)}^k))+{DSu^k_{(1:N)}}.$

\textit{2) Bounding the variation of $V_u(\cdot).$}
We have:
\begin{align}\label{eq:bound_Vu}
&V_u(u_{(1:N)}^{k+1}) \overset{(a)}{\leq} V_u(u_{(1:N)}^k) + \nabla V_u(u_{(1:N)}^k)^\tsp
(u_{(1:N)}^{k+1}-u_{(1:N)}^{k})\nonumber\\
& \quad\quad\quad\quad\quad\quad + \frac{L_{V_u}}{2} \norm{u_{(1:N)}^{k+1}-u_{(1:N)}^{k}}^2\nonumber\\
&\overset{(b)}{=} 
V_u(u_{(1:N)}^k) {-} \nabla V_u(u_{(1:N)}^k)^\tsp
F_c(u_{(1:N)}^k,\tilde x^k)\nonumber\\
& \quad\quad + \frac{L_{V_u}}{2} \norm{F_c(u_{(1:N)}^k,\tilde x^k)}^2\nonumber\\
&\overset{(c)}{\leq}
V_u(u_{(1:N)}^k) - \norm{F_c(u_{(1:N)}^k,\tilde x^k)}^2 + \frac{L_{V_u}}{2} \norm{F_c(u_{(1:N)}^k,\tilde x^k)}^2 \nonumber\\
& \quad\quad + \norm{F_c(u_{(1:N)}^k,\tilde x^k)} \norm{ \nabla V_u(u_{(1:N)}^k)^\tsp - 
F_c(u_{(1:N)}^k,\tilde x^k)} \nonumber\\
% 
% &\leq 
% V_u(u_{(1:N)}^k) 
% - \left(1 - \frac{L_{V_u}}{2}\right) \norm{F_c(u_{(1:N)}^k,\tilde x^k)}^2  \nonumber\\
% &\quad\quad + \norm{F_c(u_{(1:N)}^k,\tilde x^k)} \norm{ \nabla V_u(u_{(1:N)}^k)^\tsp - 
% F_c(u_{(1:N)}^k,x^k)}\nonumber\\ 
% 
&\overset{(d)}{\leq} 
V_u(u_{(1:N)}^k) 
- {\mu} \norm{F_c(u_{(1:N)}^k,x^k)}^2  \nonumber\\
&\quad\quad + \eta L_\Phi \norm{F_c(u_{(1:N)}^k,x^k)} \norm{ \tilde x^k}
\end{align}
where $(a)$ follows from the convexity and Lipschitz differentiability of 
$V_u,$ $(b)$ follows from of~\eqref{eq:controller_reformulation}, 
$(c)$ follows from adding and subtracting $F_c(u_{(1:N)}^k, \tilde x^k)$ 
to 
$\nabla V_u(u_{(1:N)}^k)$ and using the triangle inequality,
and $(d)$ follows from Lipschitz differentiability of $\Phi$ and definition of 
$\mu.$

\textit{3) Bounding the variation of $V_x(\cdot).$}
In the variables $\tilde x^k,$ the plant 
dynamics~\eqref{eq:interconnectedSystem-a} read as $\tilde x^{k+1} = A \tilde x^k + h(u^k_{(1:N)}) - h(u^{k+1}_{(1:N)}).$
% \begin{align*}
% \tilde x^{k+1} = A \tilde x^k + h(u^k_{(1:N)}) - h(u^{k+1}_{(1:N)}).
% \end{align*}
We then have
\begin{align}\label{eq:bound_Vx}
& V_x(\tilde x^{k+1}) = (\tilde x^k)^\tsp A^\tsp P A \tilde x^k \nonumber\\
&\quad\quad +2 (\tilde x^k)^\tsp A^\tsp P (h(u^k_{(1:N)}) - h(u^{k+1}_{(1:N)})) \nonumber\\
 &\quad\quad  + (h(u^k_{(1:N)}) - h(u^{k+1}_{(1:N)}))^\tsp P (h(u^k_{(1:N)}) - h(u^{k+1}_{(1:N)}))\nonumber\\
&\overset{(a)}{\leq} 
V_x(\tilde x^k) 
- (\tilde x^k)^\tsp Q \tilde x^k 
+ \lambda_1(P) \norm{h(u^k_{(1:N)}) - h(u^{k+1}_{(1:N)})}^2 \nonumber\\
&\quad\quad + 2 \norm{A^\tsp P} \norm{\tilde x^k} \norm{h(u^k_{(1:N)}) - h(u^{k+1}_{(1:N)})}\nonumber\\
&\overset{(b)}{\leq}
V_x(\tilde x^k) 
- {\lambda_n(Q)}\norm{\tilde x^k}^2 
+ \lambda_1(P) L_h^2 \norm{u^k_{(1:N)} - u^{k+1}_{(1:N)}}^2 \nonumber\\
&\quad\quad  + 2 L_h \norm{A^\tsp P} \norm{\tilde x^k} \norm{u^k_{(1:N)} - u^{k+1}_{(1:N)}}\nonumber\\
&\overset{(c)}{=}
V_x(\tilde x^k) 
- {\lambda_n(Q)}\norm{\tilde x^k}^2 
+ \lambda_1(P) L_h^2 \norm{F_c(u_{(1:N)}^k,x^k)}^2 \nonumber\\
&\quad\quad  + 2 L_h \norm{A^\tsp P} \norm{\tilde x^k} \norm{F_c(u_{(1:N)}^k,x^k)},
\end{align}
where $(a)$ follows from Assumption~\ref{ass:stabilityPlant},
$(b)$ from Lipschitz continuity of $h,$
and $(c)$ by applying~\eqref{eq:controller_reformulation}.

\textit{4) Combining the bounds.}
By combining~\eqref{eq:bound_Vu} and \eqref{eq:bound_Vx}, we conclude
% \begin{align*}
$
U(u_{(1:N)}^{k+1}, \tilde x^{k+1}) -
U(u_{(1:N)}^{k}, \tilde x^{k})  \leq - \xi(u_{(1:N)}^{k}, \tilde x^{k})^\tsp \Lambda \xi(u_{(1:N)}^{k}, \tilde x^{k}),
$
% \end{align*}
where
$$\xi(u_{(1:N)}^{k}, \tilde x^{k}) = \begin{bmatrix}
\norm{F_c(u_{(1:N)}^k, \tilde x^k)}\\ \norm{\tilde x^k} 
\end{bmatrix}$$
and 
\begin{align*}
\Lambda  = \begin{bmatrix}
d \frac{\mu}{\eta} - (1-d)\lambda_1(P) L_h^2
& \alpha\\
 \alpha & (1-d) \lambda_n(Q)
\end{bmatrix},
\end{align*}
where $\alpha := -\frac{1}{2} (d L_\Phi + 2 (1-d) L_h \norm{A^\tsp P}).$
It is guaranteed that $\Lambda \succ 0$ when $\eta \in (0, \bar \eta), \bar \eta = \min \{\bar \eta_1, \bar \eta_2\}$, where
\begin{align*}
\bar \eta_1 &= \frac{d\mu}{(1-d)\lambda_1(P)L_h^2},\\ \nonumber
\bar \eta_2 &= \frac{d(1-d)\mu\lambda_n(Q)}{\frac{
d^2 L_\Phi^2}{4}+(1-d)^2 L_h^2\rho +d(1-d)L_hL_\Phi\norm{A^\tsp P}},
\end{align*}
with $\rho:= \norm{A^\tsp P}^2 +\lambda_n(Q)\lambda_1(P).$
By making the following choice for the free variable $d:$
$$d=\frac{L_h^2(\norm{A^\tsp P}^2+\lambda_n(Q)\lambda_1(P))-\frac{L_\Phi L_h}{2}\sqrt{\rho}}{L_h^2(\norm{A^\tsp P}^2+\lambda_n(Q)\lambda_1(P))-\frac{L_\Phi^2}{4}},$$ we obtain the largest value of $\bar \eta_3$ that guarantees decrease of the 
storage function. For the sake of simplicity, we choose $d = 0.5$, yielding the choice $\eta_2$ and $\eta_3$ in~\eqref{eq:def_etas}. 
To conclude, since $U(u_{(1:N)}^{k}, \tilde x^{k})$ strongly decreases along 
the trajectories of~\eqref{eq:interconnectedSystem} and its minimum is the 
point $(\norm{F_c(u_{(1:N)}^k, \tilde x^k)}, \norm{\tilde x^k})=(0,0),$ by La 
Salle's invariance principle the claim follows.~
\end{proof}

Theorem~\ref{thm:convergence} shows that under a sufficiently small 
choice of the stepsize $\eta,$ the two sequences $x^k$ and 
$u^{k}_{(i)}$, describing the system's and controller's states, 
converge asymptotically to a fixed point, respectively. Notice that this convergence claim is not straightforward, since the 
proposed controller~\eqref{eq:distributed_control_algorithm} 
incorporates two simultaneous steps, a consensus one and a gradient 
step. As such, this update may oscillate or fail to converge when 
$\eta$ is chosen inadequately. The upper bounds 
$\eta_1, \eta_2, \eta_3$ depend on the various parameters of the 
system~\eqref{eq:plantModel_global} and 
optimization problem~\eqref{eq:steady_state_optimization}; also, observe that the imposed bounds on $\mu$ guarantees that 
$\bar \eta>0$.
Finally, we notice that a sufficiently small choice for $\eta$ also guarantees that there a choice of $\mu$ that satisfies 
$ 0< \mu \leq 1 - \frac{{(1- \lambda_N(W)) + \eta {L_{\Phi}}}}{2}$ (e.g., $\eta L_\Phi < 1 + \lambda_N(W)$).

Before proceeding, we present an instrumental result that will be used in the 
remainder. Based on the definition of $u_{(1:N)}^k$ and 
$\gamma (u^k_{(1:N)}, y^k)$, we 
rewrite~\eqref{eq:distributed_control_algorithm} as:
\begin{align}\label{eq:update_u_gamma}
u_{(1:N)}^{k+1} = (W \otimes I) u_{(1:N)}^k - \eta \, \gamma (u^k_{(1:N)}, y^k),
\end{align}
where $\otimes$ denotes the Kronecker product.

\begin{proposition}{\bf \textit{(Bounded gradient)}}
\label{prop:bounded_gradient}
Let the assumptions of Theorem~\ref{thm:convergence} hold. Moreover, assume 
that the initial conditions of the controller satisfy 
$u^0_{(i)}=0, \ \forall i \in \mc V.$ Then, for all $k \in \integnneg,$ \eqref{eq:update_u_gamma} satisfies:
\begin{align}
\label{eq:gradBounds}
\norm{\gamma(u^k_{(1:N)}, y^k)} \leq \sigma,
% \norm{\gamma(u^k_{(1:N)}, y^k)} \leq \sigma,
\end{align}
where $\sigma := \sqrt{
2L_{\Phi} \left( \sum_{i=1}^N \Phi_i(u^0_{(i)}, y^0)  - \sps{\Phi}{opt} \right)}$
Here, 
$\Phi^\text{opt} = \sum_{i=1}^N \Phi_i^\text{opt},$ with $\Phi_i^\text{opt} = \Phi_i(u_{(i)}^\text{opt}, y^\text{opt})$  and $(u_{(i)}^\text{opt}, y^\text{opt}) = \arg \min_{u, y} \Phi_i(u, y)$.
\end{proposition}

\begin{proof}
% Please see~\cite{AM-GB:24} for the complete proof.
To prove~\eqref{eq:gradBounds}, we rewrite
\begin{align}\label{eq:proofThm1_1}
\sum_{i=1}^N \Phi_i(u_{(i)}^k, y^k) 
&\overset{(a)}{\leq} \eta^{-1} \tilde V(u_{(1:N)}^k, \tilde x^k)\\
&\overset{(b)}{\leq} \eta^{-1} \tilde V(u_{(1:N)}^0, \tilde x^0)
\overset{(c)}{=} \sum_{i=1}^N \Phi_i(u_{(i)}^0, y^0) \nonumber
\end{align}
where $(a)$ follows from~\eqref{eq:Vu} and by using  
$\sum_{i=1}^N  \norm{u_{(i)}}^2  - \sum_{i,j=1}^N w_{ij} u_{(i)}^\tsp u_{(j)} \geq 0$ (which holds since $\beta<1$), $(b)$ 
holds from $\tilde V_u(u_{(1:N)}^{k+1}, \tilde x^{k+1}) \leq \tilde  V_u(u_{(1:N)}^k, \tilde x^k)$ (which follows by iterating the steps 
in~\eqref{eq:bound_Vu} applied to $\tilde V_u(u_{(1:N)}^{k+1}, \tilde x^{k+1})$ instead of $V_u(u_{(1:N)}^{k+1})$), and 
$(c)$ follows from~\eqref{eq:Vu_Vx_definition} using $u^0_{(1:N)}=0$.

Moreover, recall that for any differentiable convex function $g$ with 
minimizer $u^*, y^*$, and Lipschitz constant $L_{g}$, we have $g(u_a, y_a) \geq g(u_b, y_b) + \nabla g^\tsp(u_b, y_b) \begin{bmatrix}
    u_a - u_b & y_a - y_b
\end{bmatrix}^\tsp + \frac{1}{2L_{g}} \norm{\nabla g(u_a, y_a) - \nabla g(u_b, y_b)}^2$ and $\nabla g(u^*, y^*) = 0$. Then, $\norm {\nabla g(u, y)}^2 \leq 2L_{g}(g(u, y) - g^*)$, where $g^* := g(u^*, y^*)$. Using this inequality and \eqref{eq:proofThm1_1}, we 
obtain
\begin{align}\label{eq:proofThm1_2}
    & \norm{\gamma(u^k_{(1:N)}, y^k)}^2 = \sum_{i=1}^N \norm{\Pi^\tsp \nabla \Phi_i(u_{(i)}^k, y^k)}^2  \\ \nonumber
    &\quad\quad\leq \sum_{i=1}^N 2L_{\Phi}(\Phi_i(u_{(i)}^k, y^k) - \Phi_i^o) \nonumber \\ \nonumber
    &\quad\quad\leq 2L_{\Phi} \left( \sum_{i=1}^N \Phi_i(u^0_{(i)}, y^0)  
    % + \eta^{-1} L_\Phi \norm{C} \norm{\tilde x^k} 
    - \sps{\Phi}{opt} \right),
\end{align}
where $\Phi^\text{opt} = \sum_{i=1}^N \Phi_i^\text{opt},$ with $\Phi_i^\text{opt} = \Phi_i(u_{(i)}^\text{opt}, y^\text{opt})$  and $(u_{(i)}^\text{opt}, y^\text{opt}) = \arg \min_{u, y} \Phi_i(u, y)$. Note that $u_{(i)}^\text{opt}, y^\text{opt}$ exist {because of Assumption \ref{ass:convexityLocalCost} and \ref{ass:networkConnectivity}.}

\end{proof}

Proposition~\ref{prop:bounded_gradient} ensures that the sequence of gradients
$\gamma(u^k_{(1:N)}, y^k)$ is uniformly bounded. This result will be key in 
the subsequent section when characterizing the control error. Interestingly, 
unlike~\cite{JD-AA-MW:11,DJ-JX-JM:14,AN-AO:09} that assume bounded gradient, 
in our analysis our choice of stepsize ensures that gradient remains bounded.
We conclude by noting that when the initial conditions $u^0_{(1:N)}$ are 
nonzero, a uniform bound of the form~\eqref{eq:gradBounds} can still be proven, $\sigma$ needs to be modified to account for additional error terms. 

% by adjusting the step~\eqref{eq:proofThm1_2}

\subsection{Control error bounds}\label{se}
While Theorem~\ref{thm:convergence} certifies that the states sequences 
converge asymptotically, it remains to quantify explicitly the controller 
performance. This is the focus of this section. In line with~\cite{KY-QL-WY:16N}, to establish a linear 
rate of convergence, we will restrict our focus on cost functions that are 
restricted strongly convex; recall that $f: \operatorname{dom} f \rightarrow \real$ is restricted strongly convex~\cite{ML-WY:13} with modulus $\nu_f$ if 
\begin{align}\label{def:restricted_strongly_cvx}
(\nabla f(z)-\nabla f\left(z^*\right))^\tsp (z-z^*) \geq \nu_f \left\|z-z^*\right\|^2,
\end{align}
for all $z \in \operatorname{dom} f,$ where $z^*$ is such that 
$\nabla f\left(z^*\right)=0$. 
% for all $z \in \operatorname{dom} f, z^*=\operatorname{Proj}_{\mc Z^*}(z)$, 
% where $\operatorname{Proj}_{\mathcal{Z}^*}(z)$ is the projection of $z$ onto the solution set $\mathcal{Z}^*$ such that $\nabla f\left(z^*\right)=0$. 
% 
The following result is instrumental. 

\begin{lemma}{\cite[Lemma 6]{ML-WY:13}}
\label{lem:lemma6}
    Suppose that $f$ is restricted strongly convex with modulus $\nu_f$ and $\nabla f$ is Lipschitz continuous with constant $L_f$. Then, 
    \begin{align}\label{eq:c1_c2_inequality}
        &(z - z^*)^\tsp(\nabla f(z) - \nabla f(z^*) ) \\ \nonumber
        &\quad\quad\quad\quad\quad\geq c_1 \norm{\nabla f(z) - \nabla f(z^*)}^2 + c_2 \norm{z - z^*}^2, \nonumber
    \end{align}
    where $z^*$ is as in~\eqref{def:restricted_strongly_cvx}. 
Moreover, for any $\theta \in [0, 1],$ we have $c_1 = \frac{\theta}{L_f}$ and $c_2 = (1 - \theta)\nu_f.$
% \begin{align}\label{eq:definition_c1_c2}
% c_1 = \frac{\theta}{L_f}, && 
% c_2 = (1 - \theta)\nu_f.
% \end{align}
\end{lemma}

\begin{remark} \label{rem:remark4}
Notice that, if $f$ is strongly convex with modulus $\nu_f,$ then it is also 
restricted strong convexity with the same modulus~\cite{ML-WY:13}. In this case, \eqref{eq:c1_c2_inequality} holds with $c_1 = \frac{1}{\nu_f + L_f}$ and $c_2 = \frac{\nu_f L_f}{\nu_f + L_f}.$
% \begin{align}\label{eq:df_c1_c2_strong_cvx}
%  c_1 = \frac{1}{\nu_f + L_f}, && 
%  c_2 = \frac{\nu_f L_f}{\nu_f + L_f}.
% \end{align}
\QEDB
\end{remark}

The following is the second main result of this paper.

\begin{theorem}{\bf \textit{(Control error bounds)}}
\label{thm:error_bound}
Suppose the assumptions of 
Proposition~\ref{prop:bounded_gradient} hold, that 
$(u,y) \mapsto \Phi(u,y)$ is restricted strongly convex with 
modulus $\nu_\Phi$, and that $\eta \leq c_1$.
Then, \eqref{eq:interconnectedSystem} satisfies:
\begin{align}\label{eq:ultimate_bound_u}
&\norm{u_{(i)}^k - u^{*k}} \leq c_3^k \norm{u_{(i)}^0 - u^{*0}} 
+ \frac{c_4}{\sqrt{1 - c_3^2}} + \frac{\eta \sigma}{1-\beta}, 
    \end{align}
\begin{align*}
c_3^2 = 1 - \eta c_2 + \eta \delta - \eta^2 \delta c_2,  &&  
c_4^2 = \eta^3(\eta + \delta^{-1}) \frac{{L_\Phi^2}\sigma^2}{(1-\beta)^2},
\end{align*}
where $(u^{*k}, x^{*k}) := \text{Proj}_{\mc{A}^*}(\bar u^k, x^k),$ $\sigma$ is as 
in~\eqref{eq:gradBounds}, $\delta>0$ is an arbitrary constant, and 
$c_1$ and $c_2$ are as in Lemma~\ref{lem:lemma6} with $\nu_f = \nu_\Phi / N.$  Moreover, if $\Phi(u,y)$ is strongly convex, then $c_1$ and $c_2$ are as in Remark~\ref{rem:remark4}.
\QEDB\end{theorem}

\begin{proof}
We begin by proving~\eqref{eq:ultimate_bound_u}. It will be convenient to measure the 
control error relative to the average controller state:
$\bar{u}^k := \frac{1}{N} \sum_{i=1}^{N} u_{(i)}^k.$
We have
    \begin{align}\label{eq:bigIneq}
        \norm{u_{(i)}^k - u^{*k}} &\leq \norm{u_{(i)}^k - \bar u^k} + \norm{\bar u^k - u^{*k}}.
    \end{align}
The proof is organized into {two} main steps.

\noindent 
\textit{1) Bound for $\norm{u_{(i)}^k - \bar u^k}$.}
By expanding~\eqref{eq:update_u_gamma} in time:
    \begin{align}\label{eq:proofLemma2_1}
        u_{(1:N)}^k = -\eta \sum_{s=0}^{k-1} (W^{k-1-s} \otimes I) \gamma(u^s_{(1:N)}, y^s).
    \end{align}
    Next, let $\bar u_{(1:N)}^k = (\bar u^k, \cdots, \bar u^k) \in \mathbb{R}^{mN}$, 
    and notice that 
    $\bar u_{(1:N)}^k = \frac{1}{N} ((1_N1_N^\tsp) \otimes I) u_{(1:N)}^k.$
    As a result,
    \begin{align}\label{eq:proofLemma2_2}
         \Vert u_{(i)}^k &- \bar u^k \Vert \leq \norm{u_{(1:N)}^k - \bar u_{(1:N)}^k} \nonumber \\
        &= \norm{u_{(1:N)}^k - \frac{1}{N} ((1_N1_N^\tsp) \otimes I) u_{(1:N)}^k} \nonumber \\
        &= \norm{-\eta \sum_{s=0}^{k-1} (W^{k-1-s} \otimes I) \gamma(u^s_{(1:N)}, y^s) \nonumber \\
        &+ \eta \sum_{s=0}^{k-1} \frac{1}{N}((1_N1_N^\tsp W^{k-1-s}) \otimes I) \gamma(u^s_{(1:N)}, y^s)} \nonumber \\
        &\overset{(a)}{=} \norm{-\eta \sum_{s=0}^{k-1} (W^{k-1-s} \otimes I) \gamma(u^s_{(1:N)}, y^s) \nonumber \\
        &+ \eta \sum_{s=0}^{k-1} \frac{1}{N}((1_N1_N^\tsp) \otimes I) \gamma(u^s_{(1:N)}, y^s)} \nonumber \\
        &= \eta \norm{\sum_{s=0}^{k-1} \left( \left( W^{k-1-s} - \frac{1}{N} 1_N1_N^\tsp \right) \otimes I \right) \gamma(u^s_{(1:N)}, y^s)} \nonumber \\
        &\leq \eta \sum_{s=0}^{k-1} \norm{ W^{k-1-s} - \frac{1}{N} 1_N1_N^\tsp} \norm{\gamma(u^s_{(1:N)}, y^s)} \nonumber \\
        &=  \eta \sum_{s=0}^{k-1} \beta^{k-1-s} \norm{\gamma(u^s_{(1:N)}, y^s)},
    \end{align}
    where $(a)$ holds because $W$ is doubly stochastic. From 
    $\norm{\gamma(u^k_{(1:N)}, y^k)} \leq \sigma$ and $\beta < 1$, it follows that
    \begin{align}\label{eq:firstTerm}
        \norm{u_{(i)}^k - \bar u^k} &\leq \eta \sum_{s=0}^{k-1} \beta^{k-1-s} \norm{\gamma(u^s_{(1:N)}, y^s)} \leq \sum_{s=0}^{k-1} \beta^{k-1-s} \sigma \nonumber \\
        &\leq \frac{\eta \sigma}{1-\beta}.
    \end{align}

\noindent
\textit{2) Bound for $\norm{\bar u^k - u^{*k}}$.}
We will denote in compact form, $\bar e^k := \bar u^k - u^{*k}.$
    % \begin{align*}
    %     \bar e^k := \bar u^k - u^{*k}.
    % \end{align*}
To bound this term,  let
    \begin{align*}
        g(u_{(1:N)}^k, y^k) &= \frac{1}{N} \sum_{i=1}^N \Pi^\tsp \nabla \Phi_i(u_{(i)}^k, y^k), \nonumber \\
        \bar g(u_{(1:N)}^k, y^k) &= \frac{1}{N} \sum_{i=1}^N \Pi^\tsp \nabla \Phi_i(\bar u^k, y^k).
    \end{align*}
    We are interested in $g(u_{(1:N)}^k, y^k)$ because $- \eta g(u_{(1:N)}^k, y^k)$ updates $\bar u^k$. To see this, by taking the average of \eqref{eq:distributed_control_algorithm} over $i$ and noticing $W = [w_{ij}]$ is doubly stochastic, we obtain
    \begin{align} \label{eq:averageUpdate}
        \bar u^{k+1} &= \frac{1}{N} \sum_{i=1}^N u_{(i)}^{k+1} \nonumber\\
        &= \frac{1}{N} \sum_{i, j=1}^N w_{ij}u_{(j)}^{{k}}
        - \frac{\eta}{N} \sum_{i=1}^N \Pi^\tsp \nabla \Phi_i(u_{(i)}^k, y^k) \nonumber\\
        &= \bar u^k - \eta g(u_{(1:N)}^k, y^k).
    \end{align}

\noindent
Before proceeding notice that the following bound holds:
    \begin{align*}
        \norm{\Pi^\tsp (\nabla \Phi_i(u_{(i)}^k, y^k) - \nabla \Phi_i(\bar u^k, y^k))} &\leq L_{\Phi} \norm{u_{(i)}^k - \bar u^k} \\
        &\overset{(a)}{\leq} \frac{\eta \sigma L_{\Phi}}{1-\beta}
    \end{align*}
by Assumptions \ref{ass:convexityLocalCost}, \ref{ass:networkConnectivity}, and \eqref{eq:lipschitz_gradient}, and where $(a)$ follows from \eqref{eq:firstTerm}. Moreover, we also have
    \begin{align}\label{eq:boundGGbar}
        & \norm{g(u_{(1:N)}^k, y^k) - \bar g(u_{(1:N)}^k, y^k)}\nonumber\\
        &\quad\quad = \norm{\frac{1}{N} \sum_{i=1}^N \Pi^\tsp (\nabla \Phi_i(u_{(i)}^k, y^k) - \nabla \Phi_i(\bar u^k, y^k))} \nonumber\\
        &\quad\quad \leq \frac{1}{N} L_{\Phi} \sum_{i=1}^N \norm{u_{(i)}^k - \bar u^k}  \leq \frac{\eta \sigma L_{\Phi}}{1 - \beta}.
    \end{align}

\noindent
    Recalling that $(u^{*k+1},x^{*k+1}) = \text{Proj}_\mathcal{A^*}(\bar u^{k+1}, x^{k+1})$ and $\bar e^{k+1} = \bar u^{k+1} - u^{*k+1}$, we have
    \begin{align*}
        &\norm{\bar e^{k+1}}^2 \overset{(a)}{\leq} \norm{\bar u^{k+1} - u^{*k}}^2 = \norm{\bar u^k - u^{*k} - \eta g(u_{(1:N)}^k, y^k)}^2\\
        &= \norm{\bar e^k - \eta \bar g(u_{(1:N)}^k, y^k) + \eta(\bar g(u_{(1:N)}^k, y^k) - g(u_{(1:N)}^k, y^k))}^2 \nonumber \\
        &= \norm{\bar e^k - \eta \bar g(u_{(1:N)}^k, y^k)}^2 + \eta^2 \norm{\bar g(u_{(1:N)}^k, y^k) \nonumber \\
        &- g(u_{(1:N)}^k, y^k)}^2 \nonumber \\
        &+ 2\eta (\bar g(u_{(1:N)}^k, y^k) - g(u_{(1:N)}^k, y^k))^\tsp (\bar e^k - \eta \bar g(u_{(1:N)}^k, y^k)) \nonumber \\
        &\overset{(b)}{\leq} (1 + \eta \delta) \norm{\bar e^k - \eta \bar g(u_{(1:N)}^k, y^k)}^2 \\
        &\quad \quad + \eta(\eta + \delta^{-1}) \norm{\bar g(u_{(1:N)}^k, y^k) 
        - g(u_{(1:N)}^k, y^k)}^2.
    \end{align*}
$(a)$ holds since $u^*_{k+1}$ is the projection of 
$\bar u_{k+1}$ onto the optimality set, and thus for any other optimizer 
$\hat u^*_{k+1}$ we have $|\hat u^*_{k+1} - \bar u_{k+1}| \geq |u^*_{k+1} - \bar u_{k+1}|$. $(b)$ follows from 
$\pm 2 a^\tsp b \leq \delta^{-1} \norm{a}^2 + \delta \norm{b}^2$ for any 
$\delta \geq 0$.     
Next, we shall bound $\norm{\bar e^k - \eta \bar g(u_{(1:N)}^k, y^k)}^2$. Applying Lemma \eqref{lem:lemma6}, %and noticing $\red{\bar g(u, y) = \nabla f(u, y)}$, 
we have
    \begin{align*}
        &\norm{\bar e^k - \eta \bar g(u_{(1:N)}^k, y^k)}^2 = \norm{\bar e^k}^2 + \eta^2 \norm{\bar g(u_{(1:N)}^k, y^k)}^2 \nonumber \\
        &- 2\eta \bar e^{k\tsp} \bar g(u_{(1:N)}^k, y^k) \leq \norm{\bar e^k}^2 + \eta^2 \norm{\bar g(u_{(1:N)}^k, y^k)}^2 \nonumber \\
        &- \eta c_1 \norm{\bar g(u_{(1:N)}^k, y^k)}^2 - \eta c_2 \norm{\bar e^k}^2 \nonumber \\
        &= (1-\eta c_2)\norm{\bar e^k}^2 + \eta (\eta - c_1) \norm{\bar g(u_{(1:N)}^k, y^k)}^2.
    \end{align*}
    We shall pick $\eta \leq c_1$ so that $\eta (\eta - c_1) \norm{\bar g(u_{(1:N)}^k, y^k)}^2 \leq 0$. Then, from the last two inequality arrays, we have
    \begin{align*}
        &\norm{\bar e^{k+1}}^2 \leq (1 + \eta \delta)(1 - \eta c_2)\norm{\bar e^k}^2\\
        &\quad\quad\quad\quad + \eta (\eta + \delta^{-1}) \norm{\bar g(u_{(1:N)}^k, y^k) 
        - g(u_{(1:N)}^k, y^k)}^2 \nonumber \\
        &\overset{(a)}{\leq} (1 - \eta c_2 + \eta \delta - \eta^2 \delta c_2) \norm{\bar e^k}^2 + \eta^3 (\eta + \delta^{-1}) \frac{L_\Phi^2 \sigma^2}{(1-\beta)^2},
    \end{align*}
    where $(a)$ follows from~\eqref{eq:boundGGbar}. Note that if $\Phi$ is restricted strongly convex, then $c_1 c_2 = \frac{\theta (1-\theta) \nu_\Phi}{L_\Phi} < 1$ because $\theta \in [0, 1]$ and $\nu_\Phi < L_\Phi$; if $\Phi$ is strongly convex, then $c_1 c_2 = \frac{\nu_\Phi L_\Phi}{(\nu_\Phi + L_\Phi)^2} < 1$. Therefore, we have $c_1 < 1/c_2$. When $\eta < c_1$, $(1 + \eta \delta)(1 - \eta c_2) > 0$.
    
    Using $\norm{\bar e^k}^2 \leq (c_3^{k})^2 \norm{\bar e^0}^2 + \frac{1 - (c_3^{k})^2}{1 - c_3^2} c_4^2 \leq (c_3^{k})^2 \norm{\bar e^0}^2 +\frac{c_4^2}{1 - c_3^2},$
    % \begin{align*}
    %     \norm{\bar e^k}^2 \leq (c_3^{k})^2 \norm{\bar e^0}^2 + \frac{1 - (c_3^{k})^2}{1 - c_3^2} c_4^2 \leq (c_3^{k})^2 \norm{\bar e^0}^2 +\frac{c_4^2}{1 - c_3^2},
    % \end{align*}
    we get
    \begin{align}\label{eq:bound_bar_e}
        \norm{\bar e^k} \leq c_3^k \norm{\bar e^0} + \frac{c_4}{\sqrt{1 - c_3^2}}.
    \end{align}
The claim thus follows by combining~\eqref{eq:firstTerm} 
and \eqref{eq:bound_bar_e}
\end{proof}

Theorem~\ref{thm:error_bound} shows that the local agents states converge 
geometrically until reaching a neighborhood of the optimal solution. 
The size of this neighborhood depends on two quantities: 
$\frac{\eta \sigma}{1-\beta},$ which measures the asymptotic error 
due to an inexact consensus (namely, $ \norm{u_{(i)}^k - \bar u^k}$, where $\bar{u}^k := \frac{1}{N} \sum_{i=1}^{N} u_{(i)}^k$), and 
$\frac{c_4}{\sqrt{1 - c_3^2}},$  which quantifies the asymptotic error between 
the average and the optimizer (namely, $\norm{\bar u^k - u^{*k}}$). 
We conclude with the following remark, which relates $\frac{c_4}{\sqrt{1 - c_3^2}}$ explicitly with $\eta$ and $\beta.$

\begin{remark}{\bf \textit{(Refinement of bound~\eqref{eq:ultimate_bound_u})}} By choosing $\delta = \frac{c_2}{2(1-\eta c_2)},$ we have $c_3 = \sqrt{1 - \frac{\eta c_2}{2}} \in (0, 1)$ and 
    \begin{align*}
        \frac{c_4}{\sqrt{1 - c_3^2}} &= \frac{\eta L_\Phi \sigma}{1-\beta} \sqrt{\frac{\eta (\eta + \frac{2(1 - \eta c_2)}{c_2})}{\frac{\eta c_2}{2}}} = \frac{\eta L_\Phi \sigma}{1-\beta} \sqrt{\frac{4}{c_2^2} - \frac{2}{c_2} \eta} \nonumber \\
        &\leq \frac{2\eta L_\Phi \sigma}{c_2(1-\beta)} 
        = \mathcal{O} \left( \frac{\eta}{1-\beta} \right).
    \end{align*}
In this case, the local agent states converge geometrically to an 
$\mathcal{O} \left( \frac{\eta}{1-\beta} + \frac{\eta \sigma}{1-\beta}\right)$
neighborhood 
of the solution set $\mc A^*$
\QEDB\end{remark}

\section{Simulation Results}
\label{sec:simulations}
In this section, we report our numerical results. We consider a 
circular network consisting of $N = 15$ agents. The elements of 
matrices $B$, $C$, and $E$ are randomly drawn from the normal 
distribution and $A$ is chosen as a Schur stable matrix with random 
entries and circulant structure. 
We choose $n_i=1, \ \forall i,$ so that $n=N.$ We choose the mixing matrix $W$ with the same circular structure as $A$ using the Metropolis weight selection.

We apply \eqref{eq:interconnectedSystem} to:
\begin{align} \label{eq:problem}
    \underset{u \in \real^m}{\text{minimize}} ~\sum_{i=1}^N\frac{1}{2} \left(\norm{u}_{R_i}^2 + \norm{Gu+Hq-y^{ref}}_Q^2 \right),
\end{align}
where $Q = I_p$ and $R_i = \alpha_i I_m$ are the multiplied identity matrices of the corresponding dimension. $\alpha_i$ is a constant, randomly chosen from $[0.001, 0.1]$ interval for each agent. We set the desired output to $y^{ref} = 0.5\textbf{1}_p$, where $\textbf{1}_p$ is a vector of all ones. We further scale $A$ to let $\norm{A}_2 = 0.2$. The constant disturbance $q$ is generated from the standard uniform distribution. 
% Notice that the cost function in~\eqref{eq:problem} is strongly convex in this case.
\begin{figure}[t]
    % \vspace{-.3cm}
    \centering 
    \subfigure[Comparison of the proposed distributed algorithm for the problem \eqref{eq:problem} with different stepsizes.]{\includegraphics[width=1\columnwidth]{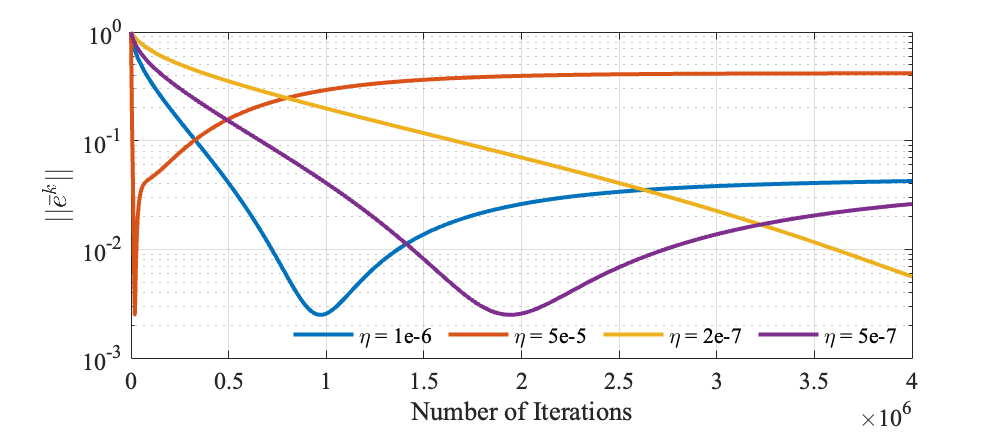}}
    % \vspace{-.3cm}
    \centering \subfigure[Inputs of the system with $\eta = 1\times10^{-6}$.]{\includegraphics[width=1\columnwidth]{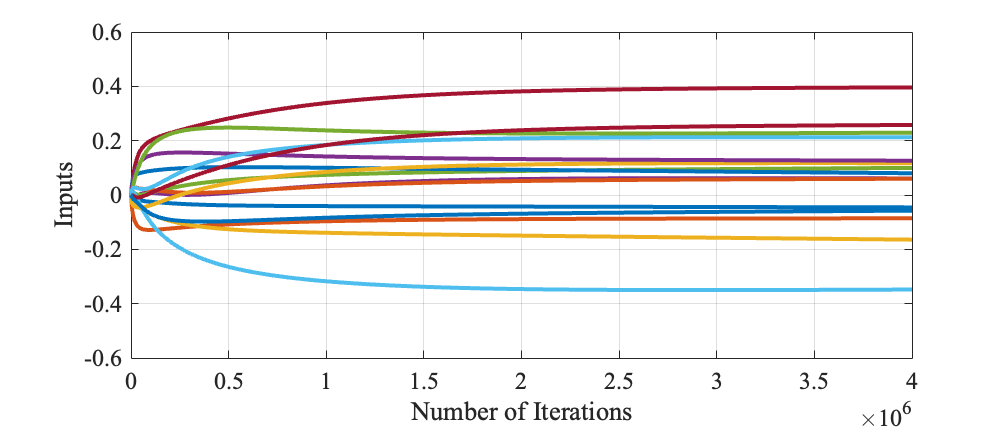}} 
    % \vspace{-.1cm}
    \centering \subfigure[Outputs of the system with $\eta = 1\times10^{-6}$.]{\includegraphics[width=1\columnwidth]{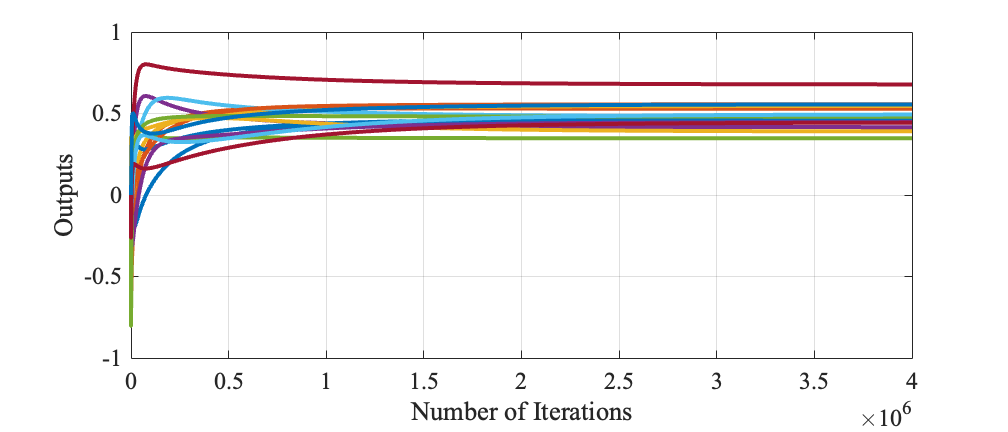}}
    \caption{Error $\bar e^k = \frac{1}{N} \sum_{i=1}^{N} \norm{u_{(i)}^k - u^{*k}}$, inputs, and outputs of the proposed decentralized algorithm with different stepsizes, where $u^{*k} = \text{Proj}_\mathcal{A^*}(\frac{1}{N} \sum_{i=1}^{N} u_{(i)}^k)$.}
    \vspace{-.4cm}
    \label{fig:sims}
\end{figure}

Fig.~\ref{fig:sims}(a) illustrates the error $\bar e^k$ for four different choices of stepsize. The simulations show that $\bar e^k$ reduces geometrically until reaching an $\mathcal{O} (\eta)$-neighborhood of the optimal point, thus validating the conclusions of Theorem~\ref{thm:error_bound}. Moreover, it shows that a smaller $\eta$ causes the algorithm to converge slowly but more accurately, as predicted by Theorem~\ref{thm:error_bound}. Fig.~\ref{fig:sims}(b) and Fig.~\ref{fig:sims}(c) illustrate the inputs and outputs of the system. Notice that the optimizer is a point that strikes a balance between tracking $y^{ref} = 0.5\textbf{1}_p$ and minimizing the control effort $\norm{u}_R^2.$ As the inputs of the system converge to the solution of~\eqref{eq:problem}, the outputs of the system are also approaching the desired output (i.e. $0.5$). Recalling the fact that DGD converges inexactly and the presence of disturbance justifies the error between the desired output and the actual system output.

\section{Conclusions} \label{sec:conclusions}
We proposed a distributed controller to solve optimal steady-state regulation 
problems while rejecting constant disturbances. The controller follows 
a distributed architecture; as such, the approach scales well with the system 
size and can be applied in cases where the individual cost functions need to be 
maintained private. Under convexity and smoothness assumptions, we showed 
that the controller state converges; under restricted strong convexity assumptions, we showed that the controller converges geometrically to a neighborhood of the optimal solution, in line with the existing literature on distributed optimization~\cite{KY-QL-WY:16N}. 
Our work opens the opportunity for several future works, including scenarios 
where the output feedback signals are also local, an investigation of 
algorithms that can ensure exact convergence, the study of constrained 
optimization objectives, and a generalization to nonlinear systems. 
\bibliographystyle{IEEEtran}
\bibliography{REFs/full_GB,REFs/GB,REFs/alias,REFs/AM_bibliography}

\begin{thebibliography}{10}
\providecommand{\url}[1]{#1}
\csname url@samestyle\endcsname
\providecommand{\newblock}{\relax}
\providecommand{\bibinfo}[2]{#2}
\providecommand{\BIBentrySTDinterwordspacing}{\spaceskip=0pt\relax}
\providecommand{\BIBentryALTinterwordstretchfactor}{4}
\providecommand{\BIBentryALTinterwordspacing}{\spaceskip=\fontdimen2\font plus
\BIBentryALTinterwordstretchfactor\fontdimen3\font minus \fontdimen4\font\relax}
\providecommand{\BIBforeignlanguage}[2]{{%
\expandafter\ifx\csname l@#1\endcsname\relax
\typeout{** WARNING: IEEEtran.bst: No hyphenation pattern has been}%
\typeout{** loaded for the language `#1'. Using the pattern for}%
\typeout{** the default language instead.}%
\else
\language=\csname l@#1\endcsname
\fi
#2}}
\providecommand{\BIBdecl}{\relax}
\BIBdecl

\bibitem{AJ-ML-PB:09}
A.~Jokic, M.~Lazar, and P.~{van den Bosch}, ``On constrained steady-state regulation: Dynamic {KKT} controllers,'' vol.~54, no.~9, pp. 2250--2254, 2009.

\bibitem{ED-AS-SB-LM:20}
E.~Dall'Anese, A.~Simonetto, S.~Becker, and L.~Madden, ``Optimization and learning with information streams: {T}ime-varying algorithms and applications,'' \emph{IEEE Signal Processing Magazine}, vol.~37, no.~3, pp. 71--83, 2020.

\bibitem{MC-ED-AB:20}
M.~{Colombino}, E.~{Dall’Anese}, and A.~{Bernstein}, ``Online optimization as a feedback controller: Stability and tracking,'' vol.~7, no.~1, pp. 422--432, 2020.

\bibitem{AH-SB-GH-FD:20}
A.~Hauswirth, S.~Bolognani, G.~Hug, and F.~D{\"o}rfler, ``Timescale separation in autonomous optimization,'' vol.~66, no.~2, pp. 611--624, 2021.

\bibitem{SM-AH-SB-GH-FD:18}
S.~Menta, A.~Hauswirth, S.~Bolognani, G.~Hug, and F.~D{\"o}rfler, ``Stability of dynamic feedback optimization with applications to power systems,'' in \emph{Annual Conf. on Communication, Control, and Computing}, Oct. 2018, pp. 136--143.

\bibitem{LL-JS-EM:18}
L.~S.~P. Lawrence, J.~W. Simpson-Porco, and E.~Mallada, ``Linear-convex optimal steady-state control,'' 2018, to appear.

\bibitem{GB-JC-JP-ED:21-tcns}
G.~Bianchin, J.~Cort\'{e}s, J.~I. Poveda, and E.~Dall'Anese, ``Time-varying optimization of {LTI} systems via projected primal-dual gradient flows,'' vol.~9, no.~1, pp. 474--486, Mar. 2022.

\bibitem{GB-JIP-ED:20-automatica}
G.~Bianchin, J.~I. Poveda, and E.~Dall'Anese, ``Online optimization of switched {LTI} systems using continuous-time and hybrid accelerated gradient flows,'' vol. 146, p. 110579, Aug. 2022.

\bibitem{FB-HD-CE:12}
F.~Brunner, H.-B. D{\"u}rr, and C.~Ebenbauer, ``Feedback design for multi-agent systems: A saddle point approach,'' 2012, pp. 3783--3789.

\bibitem{ZH-SB-JH-FD-XG:23}
Z.~He, S.~Bolognani, J.~He, F.~D{\"o}rfler, and X.~Guan, ``Model-free nonlinear feedback optimization,'' 2023, in press.

\bibitem{XC-JP-NL:22}
X.~Chen, J.~I. Poveda, and N.~Li, ``Model-free feedback constrained optimization via projected primal-dual zeroth-order dynamics,'' \emph{arXiv preprint}, 2022, arXiv:2206.11123.

\bibitem{GB-MV-JC-ED:24-tac}
G.~Bianchin, M.~Vaquero, J.~Cort\'{e}s, and E.~Dall'Anese, ``Online stochastic optimization for unknown linear systems: Data-driven synthesis and controller analysis,'' 2024, early Access.

\bibitem{AN-AO:07}
A.~Nedi{\'c} and A.~Ozdaglar, ``Distributed subgradient methods for multi-agent optimization,'' 2009, to appear.

\bibitem{KY-QL-WY:16N}
K.~Yuan, Q.~Ling, and W.~Yin, ``On the convergence of decentralized gradient descent,'' \emph{SIAM Journal on Optimization}, vol.~26, no.~3, pp. 1835--1854, 2016.

\bibitem{AN-AO-PAP:10}
A.~Nedi{\'c}, A.~Ozdaglar, and P.~A. Parrilo, ``Constrained consensus and optimization in multi-agent networks,'' vol.~55, no.~4, pp. 922--938, 2010.

\bibitem{JD-AA-MW:11}
J.~C. Duchi, A.~Agarwal, and M.~J. Wainwright, ``Dual averaging for distributed optimization: Convergence analysis and network scaling,'' vol.~57, no.~3, pp. 592--606, 2011.

\bibitem{DJ-JX-JM:14}
D.~Jakoveti{\'c}, J.~Xavier, and J.~M. Moura, ``Fast distributed gradient methods,'' vol.~59, no.~5, pp. 1131--1146, 2014.

\bibitem{AN-AO:09}
A.~Nedi{\'c} and A.~Ozdaglar, ``Distributed subgradient methods for multi-agent optimization,'' vol.~54, no.~1, pp. 48--61, 2009.

\bibitem{CX-VM-RX-EA-UK:18}
C.~Xi, V.~S. Mai, R.~Xin, E.~H. Abed, and U.~A. Khan, ``Linear convergence in optimization over directed graphs with row-stochastic matrices,'' vol.~63, no.~10, pp. 3558--3565, 2018.

\bibitem{TY-XY-JW-etal:19}
T.~Yang, X.~Yi, J.~Wu, Y.~Yuan, D.~Wu, Z.~Meng, Y.~Hong, H.~Wang, Z.~Lin, and K.~H. Johansson, ``A survey of distributed optimization,'' \emph{Annual Reviews in Control}, vol.~47, pp. 278--305, 2019.

\bibitem{TL-ZQ-YH-ZPJ:21}
T.~Liu, Z.~Qin, Y.~Hong, and Z.-P. Jiang, ``Distributed optimization of nonlinear multiagent systems: A small-gain approach,'' vol.~67, no.~2, pp. 676--691, 2021.

\bibitem{WW-ZH-GB-SB-FD:23}
W.~Wang, Z.~He, G.~Belgioioso, S.~Bolognani, and F.~D{\"o}rfler, ``Decentralized feedback optimization via sensitivity decoupling: Stability and sub-optimality,'' \emph{arXiv preprint}, 2023, arXiv:2311.09408.

\bibitem{CC-MC-JC-ED:19}
C.-Y. Chang, M.~Colombino, J.~Cort{\'e}s, and E.~Dall’Anese, ``Saddle-flow dynamics for distributed feedback-based optimization,'' vol.~3, no.~4, pp. 948--953, 2019.

\bibitem{GC-GN:22}
G.~Carnevale and G.~Notarstefano, ``Nonconvex distributed optimization via lasalle and singular perturbations,'' vol.~7, pp. 301--306, 2022.

\bibitem{HKK:96}
H.~K. Khalil, \emph{Nonlinear Systems}, 2nd~ed., 1995.

\bibitem{ML-WY:13}
M.-J. Lai and W.~Yin, ``Augmented $\ell_1$ and {N}uclear-norm models with a globally linearly convergent algorithm,'' \emph{SIAM Journal on Imaging Sciences}, vol.~6, no.~2, pp. 1059--1091, 2013.

\end{thebibliography}


\begin{thebibliography}{99}
    
    \bibitem{c1} G. O. Young, ÒSynthetic structure of industrial plastics (Book style with paper title and editor),Ó 	in Plastics, 2nd ed. vol. 3, J. Peters, Ed.  New York: McGraw-Hill, 1964, pp. 15Ð64.
    \bibitem{c2} W.-K. Chen, Linear Networks and Systems (Book style).	Belmont, CA: Wadsworth, 1993, pp. 123Ð135.
    \bibitem{c3} H. Poor, An Introduction to Signal Detection and Estimation.   New York: Springer-Verlag, 1985, ch. 4.
    \bibitem{c4} B. Smith, ÒAn approach to graphs of linear forms (Unpublished work style),Ó unpublished.
    \bibitem{c5} E. H. Miller, ÒA note on reflector arrays (Periodical styleÑAccepted for publication),Ó IEEE Trans. Antennas Propagat., to be publised.
    \bibitem{c6} J. Wang, ÒFundamentals of erbium-doped fiber amplifiers arrays (Periodical styleÑSubmitted for publication),Ó IEEE J. Quantum Electron., submitted for publication.
    \bibitem{c7} C. J. Kaufman, Rocky Mountain Research Lab., Boulder, CO, private communication, May 1995.
    \bibitem{c8} Y. Yorozu, M. Hirano, K. Oka, and Y. Tagawa, ÒElectron spectroscopy studies on magneto-optical media and plastic substrate interfaces(Translation Journals style),Ó IEEE Transl. J. Magn.Jpn., vol. 2, Aug. 1987, pp. 740Ð741 [Dig. 9th Annu. Conf. Magnetics Japan, 1982, p. 301].
    \bibitem{c9} M. Young, The Techincal Writers Handbook.  Mill Valley, CA: University Science, 1989.
    \bibitem{c10} J. U. Duncombe, ÒInfrared navigationÑPart I: An assessment of feasibility (Periodical style),Ó IEEE Trans. Electron Devices, vol. ED-11, pp. 34Ð39, Jan. 1959.
    \bibitem{c11} S. Chen, B. Mulgrew, and P. M. Grant, ÒA clustering technique for digital communications channel equalization using radial basis function networks,Ó IEEE Trans. Neural Networks, vol. 4, pp. 570Ð578, July 1993.
    \bibitem{c12} R. W. Lucky, ÒAutomatic equalization for digital communication,Ó Bell Syst. Tech. J., vol. 44, no. 4, pp. 547Ð588, Apr. 1965.
    \bibitem{c13} S. P. Bingulac, ÒOn the compatibility of adaptive controllers (Published Conference Proceedings style),Ó in Proc. 4th Annu. Allerton Conf. Circuits and Systems Theory, New York, 1994, pp. 8Ð16.
    \bibitem{c14} G. R. Faulhaber, ÒDesign of service systems with priority reservation,Ó in Conf. Rec. 1995 IEEE Int. Conf. Communications, pp. 3Ð8.
    \bibitem{c15} W. D. Doyle, ÒMagnetization reversal in films with biaxial anisotropy,Ó in 1987 Proc. INTERMAG Conf., pp. 2.2-1Ð2.2-6.
    \bibitem{c16} G. W. Juette and L. E. Zeffanella, ÒRadio noise currents n short sections on bundle conductors (Presented Conference Paper style),Ó presented at the IEEE Summer power Meeting, Dallas, TX, June 22Ð27, 1990, Paper 90 SM 690-0 PWRS.
    \bibitem{c17} J. G. Kreifeldt, ÒAn analysis of surface-detected EMG as an amplitude-modulated noise,Ó presented at the 1989 Int. Conf. Medicine and Biological Engineering, Chicago, IL.
    \bibitem{c18} J. Williams, ÒNarrow-band analyzer (Thesis or Dissertation style),Ó Ph.D. dissertation, Dept. Elect. Eng., Harvard Univ., Cambridge, MA, 1993. 
    \bibitem{c19} N. Kawasaki, ÒParametric study of thermal and chemical nonequilibrium nozzle flow,Ó M.S. thesis, Dept. Electron. Eng., Osaka Univ., Osaka, Japan, 1993.
    \bibitem{c20} J. P. Wilkinson, ÒNonlinear resonant circuit devices (Patent style),Ó U.S. Patent 3 624 12, July 16, 1990. 
    
    
    
    
    
    
    \end{thebibliography}

\addtolength{\textheight}{-12cm}   % This command serves to balance the column lengths

%%%%%%%%%%%%%%%%%%%%%%%%%%%%%%%%%%%%%%%%%%%%%%%%%%%%%%%%%%%%%%%%%%%%%%%%%%%%%%%%

%%%%%%%%%%%%%%%%%%%%%%%%%%%%%%%%%%%%%%%%%%%%%%%%%%%%%%%%%%%%%%%%%%%%%%%%%%%%%%%%

%%%%%%%%%%%%%%%%%%%%%%%%%%%%%%%%%%%%%%%%%%%%%%%%%%%%%%%%%%%%%%%%%%%%%%%%%%%%%%%%
\begin{comment}
    \section*{APPENDIX}
    
    Appendixes should appear before the acknowledgment.
    
    \section*{ACKNOWLEDGMENT}
    
    The preferred spelling of the word ÒacknowledgmentÓ in America is without an ÒeÓ after the ÒgÓ. Avoid the stilted expression, ÒOne of us (R. B. G.) thanks . . .Ó  Instead, try ÒR. B. G. thanksÓ. Put sponsor acknowledgments in the unnumbered footnote on the first page.

    %%%%%%%%%%%%%%%%%%%%%%%%%%%%%%%%%%%%%%%%%%%%%%%%%%%%%%%%%%%%%%%%%%%%%%%%%%%%%%%%
    
    References are important to the reader; therefore, each citation must be complete and correct. If at all possible, references should be commonly available publications.

\end{comment}

\end{document}